\documentclass{article}
\usepackage{amsfonts}
\usepackage{harvard}
\usepackage{amsmath}

\newtheorem{theorem}{Theorem}

\newtheorem{definition}[theorem]{Definition}

\newtheorem{remark}[theorem]{Remark}

\newenvironment{proof}[1][Proof]{\noindent\textbf{#1.} }{\ \rule{0.5em}{0.5em}}
\numberwithin{theorem}{section}
\numberwithin{equation}{section}
\input{tcilatex}

\begin{document}

\title{Jet Riemann-Lagrange Geometry and Some Applications in Theoretical
Biology}
\author{Ileana Rodica Nicola and Mircea Neagu}
\date{}
\maketitle

\begin{abstract}
The aim of this paper is to construct a natural Riemann-Lagrange
differential geometry on 1-jet spaces, in the sense of nonlinear
connections, generalized Cartan connections, d-torsions, d-curvatures, jet
electromagnetic fields and jet electromagnetic Yang-Mills energies, starting
from some given nonlinear evolution ODEs systems modelling biologic
phenomena like the cancer cell population model or the infection by human
immunodeficiency virus-type 1 (HIV-1) model.
\end{abstract}



\section{Historical geometrical aspects}

\hspace{5mm}The Riemann-Lagrange geometry [5] of the 1-jet space $J^{1}(T,M)$%
, where $T$ is a smooth \textit{\textquotedblright
multi-time\textquotedblright } manifold of dimension $p$ and $M$ is a smooth 
\textit{\textquotedblright spatial\textquotedblright } manifold of dimension 
$n$, contains many fruitful ideas for the geometrical interpretation of the
solutions of a given ODEs or PDEs system [7]. In this direction, authors
like P.J. Olver [11] or C. Udri\c{s}te [18] agreed that many applicative
problems coming from Physics [11], [17], Biology [6], [9] or Economics [18]
can be modelled on 1-jet spaces.

In such an aplicative-geometrical context, a lot of authors (G.S. Asanov
[1], D. Saunders [15], A. Vondra [19] and many others) studied the \textit{%
contravariant differential geometry} of 1-jet spaces. Moreover, proceeding
with the geometrical studies of G.S. Asanov [1], the second author of this
paper has recently elaborated that so-called the \textit{Riemann-Lagrange
geometry of 1-jet spaces} [5], which is a natural extension on 1-jet spaces
of the already well known \textit{Lagrangian geometry of the tangent bundle}
due to R. Miron and M. Anastasiei [4]. We emphasize that the
Riemann-Lagrange geometry of the 1-jet spaces allow us to regard the
solutions of a given ODEs (respectively PDEs) system as \textit{horizontal
geodesics} [17] (respectively, \textit{generalized harmonic maps} [7]) in a
convenient Riemann-Lagrange geometrical structure. In this way, it was given
a final solution for an open problem suggested by H. Poincar\'{e} [13] (%
\textit{find the geometrical structure which transforms the field lines of a
given vector field into geodesics}) and generalized by C. Udri\c{s}te [17] (%
\textit{find the geometric structure which converts the solutions of a given
first order PDEs system into some harmonic maps}).

In the following, let us present the main geometrical ideas used by C. Udri%
\c{s}te in order to solve the open problem of H. Poincar\'{e}. For more
details, please see the works [17] and [18].

For this purpose, let us consider a Riemannian manifold $(M^{n},\varphi
_{ij}(x))$ and let us fix an arbitrary vector field $X=(X^{i}(x))$ on $M$.
Obviously, the vector field $X$ produces the first order ODEs system (%
\textit{dynamical system}) 
\begin{equation}
\frac{dx^{i}}{dt}=X^{i}(x(t)),\text{ }\forall \text{ }i=\overline{1,n}.
\label{ODEs}
\end{equation}

Using the Riemannian metric $\varphi _{ij}$ and its Christoffel symbols $%
\gamma _{jk}^{i}$ and differentiating the first order ODEs system (\ref{ODEs}%
), after a convenient rearranging of the terms involved, C. Udri\c{s}te
constructed a second order prolongation (\textit{single-time geometric
dynamical system}) of the ODEs system (\ref{ODEs}), which has the form 
\begin{equation}
\frac{d^{2}x^{i}}{dt^{2}}+\gamma _{jk}^{i}\frac{dx^{j}}{dt}\frac{dx^{k}}{dt}%
=F_{j}^{i}\frac{dx^{j}}{dt}+\varphi ^{ih}\varphi _{kj}X^{j}\nabla _{h}X^{k},%
\text{ }\forall \text{ }i=\overline{1,n},  \label{SODEs}
\end{equation}%
where $\nabla $ is the Levy-Civita connection of the Riemannian manifold $%
(M,\varphi )$ and 
\begin{equation*}
F_{j}^{i}=\nabla _{j}X^{i}-\varphi ^{ih}\varphi _{kj}\nabla _{h}X^{k}
\end{equation*}%
is a $(1,1)$-tensor field which represents the \textit{helicity }of the
vector field $X$.

It is easy to see that any solution of class $C^{2}$ of the first order ODEs
system (\ref{ODEs}) is also a solution for the second order ODEs system (\ref%
{SODEs}). Conversely, this statement is not true.

\begin{remark}
The importance of the second order ODEs system (\ref{SODEs}) comes from its
equivalency with the Euler-Lagrange equations of that so-called the \textbf{%
least squares Lagrangian function}%
\begin{equation*}
L_{ls}:TM\rightarrow \mathbb{R}_{+},
\end{equation*}%
given by 
\begin{equation}
L_{ls}(x,y)=\frac{1}{2}\varphi _{ij}(x)\left[ y^{i}-X^{i}(x)\right] \left[
y^{j}-X^{j}(x)\right] .  \label{LS1}
\end{equation}

Note that the field lines of class $C^{2}$ of the vector field $X$ are the 
\textit{global minimum points} of the \textbf{least squares energy action}%
\begin{equation*}
\mathbb{E}_{ls}(x(t))=\int_{a}^{b}L_{ls}(x^{k}(t),\dot{x}^{k}(t))dt.
\end{equation*}

As a conclusion, the field lines of class $C^{2}$ of the vector field $X$
are solutions of the Euler-Lagrange equations produced by $L_{ls}.$ Because
the Euler-Lagrange equations of $L_{ls}$ are exactly the equations (\ref%
{SODEs}), C. Udri\c{s}te claims that the solutions of class $C^{2}$ of the
first order ODEs system (\ref{ODEs}) are \textbf{horizontal geodesics} on
the \textbf{Riemann-Lagrange manifold} [17]%
\begin{equation*}
(\mathbb{R}\times M,1+\varphi ,N(_{1}^{i})_{j}=\gamma
_{jk}^{i}y^{k}-F_{j}^{i}).
\end{equation*}
\end{remark}

\section{Riemann-Lagrange geometrical background on 1- jet spaces}

\hspace{4mm} In this Section, we regard the given first order nonlinear ODEs
system (\ref{ODEs}) as an ordinary differential system on an 1-jet space $%
J^{1}(T,\mathbb{R}^{n})$, where $T\subset \mathbb{R}$. Moreover, starting
from Udri\c{s}te's geometrical ideas, we construct some jet Riemann-Lagrange
geometrical objects (nonlinear connections, generalized Cartan connections,
d-torsions, d-curvatures, jet electromagnetic fields and jet electromagnetic
Yang-Mills energies) which, in our opinion, characterize from a geometrical
point of view the given nonlinear ODEs system of order one.

In this direction, let $T=[a,b]\subset \mathbb{R}$ be a compact interval of
the set of real numbers and let us consider the jet fibre bundle of order
one 
\begin{equation*}
J^{1}(T,\mathbb{R}^{n})\rightarrow T\times \mathbb{R}^{n},\text{ }n\geq 2,
\end{equation*}%
whose local coordinates $(t,x^{i},x_{1}^{i}),$ $i=\overline{1,n},$ transform
by the rules 
\begin{equation}
\widetilde{t}=\widetilde{t}(t),\text{ }\widetilde{x}^{i}=\widetilde{x}%
^{i}(x^{j}),\text{ }\widetilde{x}_{1}^{i}=\frac{\partial \widetilde{x}^{i}}{%
\partial x^{j}}\frac{dt}{d\widetilde{t}}\cdot x_{1}^{j}.
\label{tr_rules_jet}
\end{equation}

\begin{remark}
From a physical point of view, the coordinate $t$ has the physical meaning
of \textbf{relativistic time}, the coordinate $x=(x^{i})_{i=\overline{1,n}}$
represents the \textbf{spatial coordinate} and the coordinate $%
y=(x_{1}^{i})_{i=\overline{1,n}}$ has the physical meaning of \textbf{%
direction }or \textbf{relativistic velocity}. Thus, the coordinate $y$ is
intimately connected with the physical concept of \textbf{anisotropy}.
\end{remark}

Let us consider $X=\left( X_{(1)}^{(i)}(x^{k})\right) $ be an arbitrary 
\textit{d-tensor field} on the 1-jet space $J^{1}(T,\mathbb{R}^{n})$, whose
local components transform by the rules 
\begin{equation*}
\widetilde{X}_{(1)}^{(i)}=\frac{\partial \widetilde{x}^{i}}{\partial x^{j}}%
\frac{dt}{d\widetilde{t}}\cdot X_{(1)}^{(j)}.
\end{equation*}

Clearly, the d-tensor field $X$ produces the jet ODEs system of order one (%
\textit{jet dynamical system}) 
\begin{equation}
x_{1}^{i}=X_{(1)}^{(i)}(x^{k}(t)),\text{ }\forall \text{ }i=\overline{1,n},
\label{DEs1}
\end{equation}%
where $x(t)=(x^{i}(t))$ is an unknown curve on $\mathbb{R}^{n}$ (i. e., a 
\textit{jet field line} of the d-tensor field $X$) and we use the notation%
\begin{equation*}
x_{1}^{i}\overset{not}{=}\frac{dx^{i}}{dt},\text{ }\forall \text{ }i=%
\overline{1,n}.
\end{equation*}

\begin{remark}
The main and refined difference between the ODEs systems (\ref{ODEs}) and (%
\ref{DEs1}), which have the same form, consists only in their invariance
transformation groups. Thus, the ODEs system (\ref{ODEs}) is invariant under
the transformation group $\widetilde{x}^{i}=\widetilde{x}^{i}(x^{j})$ while
the ODEs system (\ref{DEs1}) is invariant under the transformation group $%
\widetilde{t}=\widetilde{t}(t)$, $\widetilde{x}^{i}=\widetilde{x}^{i}(x^{j})$
which does not ignore the temporal reparametrizations.
\end{remark}

Now, let us consider the Euclidian structures $(T,1)$ and $(\mathbb{R}%
^{n},\delta _{ij})$, where $\delta _{ij}$ are the Kronecker symbols. Using
as a pattern the Udri\c{s}te's geometrical ideas, we underline that the jet
first order ODEs system (\ref{DEs1}) automatically produces the \textit{jet
least squares Lagrangian function} 
\begin{equation*}
JL_{ls}:J^{1}(T,\mathbb{R}^{n})\rightarrow \mathbb{R}_{+},
\end{equation*}%
expressed by 
\begin{equation}
JL_{ls}(x^{k},x_{1}^{k})=\sum_{i=1}^{n}\left[ x_{1}^{i}-X_{(1)}^{(i)}(x)%
\right] ^{2}.  \label{JetLS}
\end{equation}

Because the \textit{global minimum points} of the \textit{jet least squares
energy action} 
\begin{equation*}
\mathbb{JE}_{ls}(c(t))=\int_{a}^{b}JL_{ls}\left( x^{k}(t),\frac{dx^{k}}{dt}%
\right) dt
\end{equation*}%
are exactly the solutions of class $C^{2}$ of the jet first order ODEs
system (\ref{DEs1}), it follows that the solutions of class $C^{2}$ of the
jet dynamical system (\ref{DEs1}) verify the second order \textit{%
Euler-Lagrange equations} produced by $JL_{ls}$ (\textit{jet geometric
dinamics}), namely%
\begin{equation}
\frac{\partial \left[ JL_{ls}\right] }{\partial x^{i}}-\frac{d}{dt}\left( 
\frac{\partial \left[ JL_{ls}\right] }{\partial x_{1}^{i}}\right) =0,\text{ }%
\forall \text{ }i=\overline{1,n}.  \label{E-L-J}
\end{equation}

In conclusion, because of all the arguments exposed above we believe that we
may regard the jet least squares Lagrangian function $JL_{ls}$ as a natural
geometrical substitut on the 1-jet space $J^{1}(T,\mathbb{R}^{n})$ for the
jet first order ODEs system (\ref{DEs1}).

\begin{remark}
A Riemann-Lagrange geometry on $J^{1}(T,\mathbb{R}^{n})$ produced by the jet
least squares Lagrangian function $JL_{ls}$, via its second order
Euler-Lagrange equations (\ref{E-L-J}), in the sense of nonlinear
connection, generalized Cartan connection, d-torsions, d-curvatures, jet
electromagnetic field and jet Yang-Mills electromagnetic energy, is now
completely done in the works [5], [6] and [7].
\end{remark}

In this geometric background, we introduce the following concept:

\begin{definition}
Any geometrical object on the 1-jet space $J^{1}(T,\mathbb{R}^{n})$, which
is produced by the jet least squares Lagrangian function $JL_{ls}$, via its
Euler-Lagrange equations (\ref{E-L-J}), is called \textbf{geometrical object
produced by the jet first order ODEs system (\ref{DEs1})}.
\end{definition}

In this context, we give the following geometrical result (this is proved in
the works [6] and [9] and, for the multi-time general case, in the paper
[7]) which characterizes the jet first order ODEs system (\ref{DEs1}). For
all details, the reader is invited to consult the book [5].

\begin{theorem}
\label{MainTh}(i) The \textbf{canonical nonlinear connection on }$J^{1}(T,%
\mathbb{R}^{n})$\textbf{\ produced by the jet first order ODEs system (\ref%
{DEs1})} is%
\begin{equation*}
\Gamma =\left( 0,N_{(1)j}^{(i)}\right) ,
\end{equation*}%
whose local components $N_{(1)j}^{(i)}$ are the entries of the matrix 
\begin{equation*}
N_{(1)}=\left( N_{(1)j}^{(i)}\right) _{i,j=\overline{1,n}}=-\frac{1}{2}\left[
J\left( X_{(1)}\right) -\text{ }^{T}J\left( X_{(1)}\right) \right] ,
\end{equation*}%
where 
\begin{equation*}
J\left( X_{(1)}\right) =\left( \frac{\partial X_{(1)}^{(i)}}{\partial x^{j}}%
\right) _{i,j=\overline{1,n}}
\end{equation*}%
is the Jacobian matrix.

(ii) All adapted components of the \textbf{canonical generalized Cartan
connection }$C\Gamma $\textbf{\ produced by the jet first order ODEs system (%
\ref{DEs1})} vanish.

(iii) The effective adapted components $R_{(1)jk}^{(i)}$ of the \textbf{%
torsion} d-tensor \textbf{T} of the canonical generalized Cartan connection $%
C\Gamma $ \textbf{produced by the jet first order ODEs system (\ref{DEs1})}
are the entries of the matrices 
\begin{equation*}
R_{(1)k}=\frac{\partial }{\partial x^{k}}\left[ N_{(1)}\right] ,\text{ }%
\forall \text{ }k=\overline{1,n},
\end{equation*}%
where 
\begin{equation*}
R_{(1)k}=\left( R_{(1)jk}^{(i)}\right) _{i,j=\overline{1,n}},\text{ }\forall 
\text{ }k=\overline{1,n}.
\end{equation*}

(iv) All adapted components of the \textbf{curvature} d-tensor \textbf{R} of
the canonical generalized Cartan connection $C\Gamma $ \textbf{produced by
the jet first order ODEs system (\ref{DEs1})} vanish.

(v) The \textbf{geometric electromagnetic distinguished 2-form produced by
the jet first order ODEs system (\ref{DEs1})} has the expression 
\begin{equation*}
\mathbf{F}=F_{(i)j}^{(1)}\delta x_{1}^{i}\wedge dx^{j},
\end{equation*}%
where 
\begin{equation*}
\delta x_{1}^{i}=dx_{1}^{i}+N_{(1)k}^{(i)}dx^{k},\text{ }\forall \text{ }i=%
\overline{1,n},
\end{equation*}%
and the adapted components $F_{(i)j}^{(1)}$ are the entries of the matrix 
\begin{equation*}
F^{(1)}=\left( F_{(i)j}^{(1)}\right) _{i,j=\overline{1,n}}=-N_{(1)}.
\end{equation*}

(vi) The adapted components $F_{(i)j}^{(1)}$ of the geometric
electromagnetic d-form $\mathbf{F}$ produced by the jet first order ODEs
system (\ref{DEs1})\textbf{\ }verify the \textbf{generalized Maxwell
equations} 
\begin{equation*}
\sum_{\{i,j,k\}}F_{(i)j||k}^{(1)}=0,
\end{equation*}%
where $\sum_{\{i,j,k\}}$ represents a cyclic sum and 
\begin{equation*}
F_{(i)j||k}^{(1)}=\frac{\partial F_{(i)j}^{(1)}}{\partial x^{k}}
\end{equation*}%
means the horizontal local covariant derivative produced by the Berwald
connection $B\Gamma _{0}$ on $J^{1}(T,\mathbb{R}^{n}).$ For more details,
please consult [5].

(vii) The \textbf{geometric jet Yang-Mills energy produced by the jet first
order ODEs system (\ref{DEs1})} is given by the formula 
\begin{equation*}
\mathbf{EYM}(x)=\frac{1}{2}\cdot Trace\left[ F^{(1)}\cdot \text{ }^{T}F^{(1)}%
\right] =\sum_{i=1}^{n-1}\sum_{j=i+1}^{n}\left[ F_{(i)j}^{(1)}\right] ^{2}.
\end{equation*}
\end{theorem}

In the next Sections, we apply the above jet Riemann-Lagrange geometrical
results to certain evolution equations from Theoretical Biology that govern
two of the most actual diseases of our times, namely the spread of cancer
cells in vivo and the infection by human immunodeficiency virus-type 1
(HIV-1). We sincerely hope that our geometrical approach of these evolution
equations to give useful mathematical informations for biologists.

\begin{remark}
For more geometrical methods applied to mathematical models coming from
Theoretical Biology, the reader is invited to consult the book [9].
\end{remark}

\section{Jet Riemann-Lagrange geometry for a cancer cell population model in
biology}

\hspace{5mm}The mathematical model of cancer cell population, which consists
of a two dimensional system of ODEs with four parameters, was introduced in
2006 by Garner et al. [2].

It is well known that cancer cell populations consist of a combination of
proliferating, quiescent and dead cells that determine tumor growth or
cancer spread [3]. Moreover, recent research in cancer progression and
treatment indicates that many forms of cancer arise from one abnormal cell
or a small subpopulation of abnormal cells [14]. These cells, which support
cancer growth and spread are called \textit{cancer stem cells} (CSCs).
Targeting these CSCs is crucial because they display many of the same
characteristics as healthy stem cells, and they have the capacity of
initiating new tumors after long periods of remmision. The understanding of
cancer mechanism could have a significant impact on cancer treatment
approaches as it emphasizes the importance of targeting diverse cell
subpopulations at a specific stage of development.

The nondimensionalized model introduced by Garner et al. is based on a
system of Solyanik et al. [16], which starts from the following assumptions:

\begin{enumerate}
\item the cancer cell population consists of proliferating and quiescent
(resting) cells;

\item the cells can lose their ability to divide under certain conditions
and then transit from the proliferating to the resting state;

\item resting cells can either return to the proliferating state or die.
\end{enumerate}

The dynamical system has two state variables, namely $P$ - \textit{the
number of proliferating cells} and $Q$ - \textit{the number of quiescent
cells}, and their evolution in time is described by the following
differential equations (\textit{cancer cell population flow}) 
\begin{equation}
\left\{ 
\begin{array}{l}
\dfrac{dP}{dt}=P-P(P+Q)+F(P,Q),\medskip \\ 
\dfrac{dQ}{dt}=-rQ+aP(P+Q)-F(P,Q),%
\end{array}%
\right.  \label{cancer}
\end{equation}%
\begin{equation*}
F(P,Q)=\dfrac{hPQ}{1+kP^{2}},\ r=\frac{d}{b},\ h=\frac{A}{ac},\ k=\frac{%
Bb^{2}}{c^{2}},
\end{equation*}

where

\begin{itemize}
\item $a$ - is a dimensionless constant that measures the relative nutrient
uptake by resting and proliferating cells;

\item $b$ - is the rate of cell division of the proliferating cells;

\item $c$ - depends on the intensity of consumption by proliferating cells
and gives the magnitude of the rate of cell transition from the
proliferating stage to the resting stage in per cell per day;

\item $d$ - is the rate of cell death of the resting cells (per day);

\item $A$ - represents the initial rate of increase in the intensity of cell
transition from the quiescent to proliferating state at small $P$;

\item $A/B$ - represents the rate of decrease in the intensity of cell
transition from the quiescent to proliferating state when $P$ becomes larger.
\end{itemize}

\medskip The Riemann-Lagrange geometrical behavior on the 1-jet space $%
J^{1}(T,\mathbb{R}^{2})$ of the \textit{cancer cell population flow} is
described in the following result:

\begin{theorem}
(i) The \textbf{canonical nonlinear connection on }$J^{1}(T,\mathbb{R}^{2})$%
\textbf{\ produced by the cancer cell population flow (\ref{cancer})} has
the local components 
\begin{equation*}
\hat{\Gamma}=\left( 0,\hat{N}_{(1)j}^{(i)}\right) ,\text{ }i,j=\overline{1,2}%
,
\end{equation*}%
where, if 
\begin{equation*}
F_{P}=\dfrac{hQ\left( 1-kP^{2}\right) }{\left( 1+kP^{2}\right) ^{2}}\text{
and }F_{Q}=\dfrac{hP}{1+kP^{2}}
\end{equation*}%
are the first partial derivatives of the function $F$, then we have%
\begin{equation*}
\begin{array}{l}
\hat{N}_{(1)1}^{(1)}=\hat{N}_{(1)2}^{(2)}=0,\medskip \\ 
\hat{N}_{(1)2}^{(1)}=-\hat{N}_{(1)1}^{(2)}=\dfrac{1}{2}\left[ \left(
2a+1\right) P+aQ-\left( F_{P}+F_{Q}\right) \right] =\medskip \\ 
=\dfrac{1}{2}\left[ \left( 2a+1\right) P+aQ-\dfrac{hQ\left( 1-kP^{2}\right) 
}{\left( 1+kP^{2}\right) ^{2}}-\dfrac{hP}{1+kP^{2}}\right] .%
\end{array}%
\end{equation*}

(ii) All adapted components of the \textbf{canonical generalized Cartan
connection }$C\hat{\Gamma}$\textbf{\ produced by the cancer cell population
flow (\ref{cancer})} vanish.

(iii) All adapted components of the \textbf{torsion} d-tensor \textbf{\^{T}}
of the canonical generalized Cartan connection $C\hat{\Gamma}$ \textbf{%
produced by the cancer cell population flow (\ref{cancer})} are zero, except 
\begin{equation*}
\begin{array}{l}
\hat{R}_{(1)21}^{(1)}=-\hat{R}_{(1)11}^{(2)}=a+\dfrac{1}{2}\left(
1-F_{PP}-F_{PQ}\right) ,\medskip \\ 
\hat{R}_{(1)22}^{(1)}=-\hat{R}_{(1)12}^{(2)}=\dfrac{1}{2}\left(
a-F_{PQ}-F_{QQ}\right) =\dfrac{1}{2}\left( a-F_{PQ}\right) ,%
\end{array}%
\end{equation*}%
where%
\begin{equation*}
F_{PP}=-\dfrac{2hkPQ\left( 3-kP^{2}\right) }{\left( 1+kP^{2}\right) ^{3}},%
\text{ }F_{PQ}=\dfrac{h\left( 1-kP^{2}\right) }{\left( 1+kP^{2}\right) ^{2}}%
\text{ and }F_{QQ}=0
\end{equation*}%
are the second partial derivatives of the function $F$.

(iv) All adapted components of the \textbf{curvature} d-tensor \textbf{\^{R}}
of the canonical generalized Cartan connection $C\hat{\Gamma}$ \textbf{%
produced by the cancer cell population flow (\ref{cancer})} vanish.

(v) The \textbf{geometric electromagnetic distinguished 2-form produced by
the cancer cell population flow (\ref{cancer})} has the expression 
\begin{equation*}
\mathbf{\hat{F}}=\hat{F}_{(i)j}^{(1)}\delta x_{1}^{i}\wedge dx^{j},
\end{equation*}%
where 
\begin{equation*}
\delta x_{1}^{i}=dx_{1}^{i}+\hat{N}_{(1)k}^{(i)}dx^{k},\text{ }\forall \text{
}i=\overline{1,2},
\end{equation*}%
and the adapted components $\hat{F}_{(i)j}^{(1)},$ $i,j=\overline{1,2},$ are
given by%
\begin{equation*}
\begin{array}{l}
\hat{F}_{(1)1}^{(1)}=\hat{F}_{(2)2}^{(1)}=0,\medskip \\ 
\hat{F}_{(2)1}^{(1)}=-\hat{F}_{(1)2}^{(1)}=\dfrac{1}{2}\left[ \left(
2a+1\right) P+aQ-\left( F_{P}+F_{Q}\right) \right] =\medskip \\ 
=\dfrac{1}{2}\left[ \left( 2a+1\right) P+aQ-\dfrac{hQ\left( 1-kP^{2}\right) 
}{\left( 1+kP^{2}\right) ^{2}}-\dfrac{hP}{1+kP^{2}}\right] .%
\end{array}%
\end{equation*}

(vi) The \textbf{biologic geometrical Yang-Mills energy produced by the
cancer cell population flow (\ref{cancer})} is given by the formula 
\begin{equation*}
\mathbf{EYM}^{\text{cancer}}(P,Q)=\frac{1}{4}\left[ \left( 2a+1\right) P+aQ-%
\dfrac{hQ\left( 1-kP^{2}\right) }{\left( 1+kP^{2}\right) ^{2}}-\dfrac{hP}{%
1+kP^{2}}\right] ^{2}.
\end{equation*}
\end{theorem}

\begin{proof}
We regard the cancer cell population flow (\ref{cancer}) as a particular
case of the jet first order ODEs system (\ref{DEs1}) on the 1-jet space $%
J^{1}(T,\mathbb{R}^{2})$, with 
\begin{equation*}
n=2,\text{ }x^{1}=P,\text{ }x^{2}=Q
\end{equation*}%
and 
\begin{eqnarray*}
X_{(1)}^{(1)}(x^{1},x^{2}) &=&x^{1}-x^{1}(x^{1}+x^{2})+F(x^{1},x^{2})\text{ ,%
} \\
X_{(1)}^{(2)}(x^{1},x^{2}) &=&-rx^{2}+ax^{1}(x^{1}+x^{2})-F(x^{1},x^{2}).
\end{eqnarray*}

Now, using the Theorem \ref{MainTh} and taking into account that we have the
Jacobian matrix 
\begin{eqnarray*}
J\left( X_{(1)}\right) &=&\left( 
\begin{array}{cc}
1-2P-Q+F_{P} & -P+F_{Q} \\ 
2aP+aQ-F_{P} & -r+aP-F_{Q}%
\end{array}%
\right) \\
&=&\left( 
\begin{array}{cc}
1-2P-Q+\dfrac{hQ\left( 1-kP^{2}\right) }{\left( 1+kP^{2}\right) ^{2}} & -P+%
\dfrac{hP}{1+kP^{2}} \\ 
2aP+aQ-\dfrac{hQ\left( 1-kP^{2}\right) }{\left( 1+kP^{2}\right) ^{2}} & 
-r+aP-\dfrac{hP}{1+kP^{2}}%
\end{array}%
\right) ,
\end{eqnarray*}%
we obtain what we were looking for.
\end{proof}

\begin{remark}[Open problem]
The \textbf{Yang-Mills biologic e\-ner\-ge\-ti\-cal \linebreak curves of
constant level produced by the cancer cell population flow (\ref{cancer})},
which are different by the empty set, are in the plane $POQ$ the curves of
implicit equations 
\begin{equation*}
\mathcal{C}_{C}:\left[ \left( 2a+1\right) P+aQ-\dfrac{hQ\left(
1-kP^{2}\right) }{\left( 1+kP^{2}\right) ^{2}}-\dfrac{hP}{1+kP^{2}}\right]
^{2}=4C,
\end{equation*}%
where $C\geq 0.$ For instance, the \textbf{zero Yang-Mills biologic
e\-ner\-ge\-ti\-cal curve produced by the cancer cell population flow (\ref%
{cancer})} is in the plane $POQ$ the graph of a rational function:%
\begin{equation*}
\mathcal{C}_{0}:Q=\frac{P\left( 1+kP^{2}\right) \left[ h-\left( 2a+1\right)
\left( 1+kP^{2}\right) \right] }{a\left( 1+kP^{2}\right) ^{2}-h\left(
1-kP^{2}\right) }.
\end{equation*}

As a possible opinion, we consider that if the cancer cell population flow
does not generate any Yang-Mills biologic energies, then it is to be
expected that the variables $P$ and $Q$ vary along the rational curve $%
\mathcal{C}_{0}$. Otherwise, if the cancer cell population flow generates an
Yang-Mills biologic energy, then it is possible that the shapes of the
constant Yang-Mills biologic e\-ner\-ge\-ti\-cal curves $\mathcal{C}_{C}$ to
offer useful interpretations for biologists.
\end{remark}

\section{The jet Riemann-Lagrange geometry of the infection by human
immunodeficiency virus (HIV-1) evolution model}

\hspace{4mm}It is well known that the major target of HIV infection is a
class of lymphocytes, or white blood cells, known as $CD4^{+}$ T cells.
These cells secrete growth and differentiations factors that are required by
other cell populations in the immune system, and hence these cells are also
called \textit{\textquotedblright helper T cells\textquotedblright }. After
becoming infected, the $CD4^{+}$ T cells can produce new HIV virus particles
(or virions) so, in order to model HIV infection it was introduced a
population of uninfected target cells $T$, and productively infected cells $%
T^{\ast }$.

Over the past decade, a number of models have been developed to describe the
immune system, its interaction with HIV, and the decline in $CD4^{+}$ T
cells. We propose for geometrical investigation a model that incorporates
viral production (for more details, please see [8], [12]). This mathematical
model of infection by HIV-1 relies on the variables $T(t)$ - \textit{the
population of uninfected target cells}, $T^{\ast }(t)$ - \textit{the
population of productively infected cells}, and $V(t)$ - \textit{the HIV-1
virus}, whose evolution in time is given by the \textit{HIV-1 flow} [12] 
\begin{equation}
\left\{ 
\begin{array}{l}
\dfrac{dT}{dt}=s+(p-d)T-\dfrac{pT^{2}}{m}-kVT\medskip \\ 
\dfrac{dT^{\ast }}{dt}=kTV-\delta T^{\ast }\medskip \\ 
\dfrac{dV}{dt}=n\delta T^{\ast }-cV,%
\end{array}%
\right.  \label{HIV}
\end{equation}%
where

\begin{itemize}
\item $s$ represents the rate at which new $T$ cells are created from
sources within the body, such as thymus;

\item $p$ is the maximum proliferation rate of $T$ cells;

\item $d$ is the death rate per $T$ cells;

\item $\delta $ represents the death rate for infected cells $T^{*}$;

\item $m$ is the $T$ cells population density at which proliferation shuts
off;

\item $k$ is the infection rate;

\item $n$ represents the total number of virions produced by a cell during
its lifetime;

\item $c$ is the rate of clearance of virions.
\end{itemize}

In what follows, we apply our jet Riemann-Lagrange geometrical results to
the \textit{HIV-1 flow} (\ref{HIV}) regarded on the 1-jet space $J^{1}(T,%
\mathbb{R}^{3})$. In this context, we obtain:

\begin{theorem}
(i) The \textbf{canonical nonlinear connection on }$J^{1}(T,\mathbb{R}^{3})$%
\textbf{\ produced by the HIV-1 flow (\ref{HIV})} has the local components 
\begin{equation*}
\check{\Gamma}=\left( 0,\check{N}_{(1)j}^{(i)}\right) ,\text{ }i,j=\overline{%
1,3},
\end{equation*}%
where $\check{N}_{(1)j}^{(i)}$ are the entries of the matrix 
\begin{equation*}
\check{N}_{(1)}=-\frac{1}{2}\left( 
\begin{array}{ccc}
0 & -kV & -kT \\ 
kV & 0 & kT-n\delta \\ 
kT & -kT+n\delta & 0%
\end{array}%
\right) .
\end{equation*}

(ii) All adapted components of the \textbf{canonical generalized Cartan
connection }$C\check{\Gamma}$\textbf{\ produced by the HIV-1 flow (\ref{HIV})%
} vanish.

(iii) All adapted components of the \textbf{torsion} d-tensor \textbf{\v{T}}
of the canonical generalized Cartan connection $C\check{\Gamma}$ \textbf{%
produced by the HIV-1 flow (\ref{HIV})} vanish, except the entries of the
matrices 
\begin{equation*}
\check{R}_{(1)1}=\left( 
\begin{array}{ccc}
0 & 0 & k/2 \\ 
0 & 0 & -k/2 \\ 
-k/2 & k/2 & 0%
\end{array}%
\right)
\end{equation*}%
and 
\begin{equation*}
\check{R}_{(1)3}=\left( 
\begin{array}{ccc}
0 & k/2 & 0 \\ 
-k/2 & 0 & 0 \\ 
0 & 0 & 0%
\end{array}%
\right) ,
\end{equation*}%
where 
\begin{equation*}
\check{R}_{(1)k}=\left( \check{R}_{(1)jk}^{(i)}\right) _{i,j=\overline{1,3}},%
\text{ }\forall \text{ }k\in \left\{ 1,3\right\} .
\end{equation*}

(iv) All adapted components of the \textbf{curvature} d-tensor \textbf{\v{R}}
of the canonical generalized Cartan connection $C\check{\Gamma}$ \textbf{%
produced by the HIV-1 flow (\ref{HIV})} vanish.

(v) The \textbf{geometric electromagnetic distinguished 2-form produced by
the HIV-1 flow (\ref{HIV})} has the expression 
\begin{equation*}
\mathbf{\check{F}}=\check{F}_{(i)j}^{(1)}\delta x_{1}^{i}\wedge dx^{j},
\end{equation*}%
where 
\begin{equation*}
\delta x_{1}^{i}=dx_{1}^{i}+\check{N}_{(1)k}^{(i)}dx^{k},\text{ }\forall 
\text{ }i=\overline{1,3},
\end{equation*}%
and the adapted components $\check{F}_{(i)j}^{(1)}$ $i,j=\overline{1,3},$
are the entries of the matrix 
\begin{equation*}
\check{F}^{(1)}=\frac{1}{2}\left( 
\begin{array}{ccc}
0 & -kV & -kT \\ 
kV & 0 & kT-n\delta \\ 
kT & -kT+n\delta & 0%
\end{array}%
\right) .
\end{equation*}

(vi) The \textbf{biologic geometric Yang-Mills energy produced by the HIV-1
flow (\ref{HIV})} is given by the formula 
\begin{equation*}
\mathbf{EYM}^{\text{HIV-1}}(T,T^{\ast },V)=\frac{1}{4}\left[
k^{2}(V^{2}+T^{2})+(kT-n\delta )^{2}\right].
\end{equation*}
\end{theorem}

\begin{proof}
Consider the HIV-1 flow (\ref{HIV}) as a particular case of the jet first
order ODEs system (\ref{DEs1}) on the 1-jet space $J^{1}(T,\mathbb{R}^{3})$,
with 
\begin{equation*}
n=3,\text{ }x^{1}=T,\text{ }x^{2}=T^{\ast },\text{ }x^{3}=V
\end{equation*}%
and%
\begin{equation*}
\begin{array}{l}
X_{(1)}^{(1)}(x^{1},x^{2},x^{3})=s+(p-d)x^{1}-\dfrac{p}{m}%
(x^{1})^{2}-kx^{3}x^{1},\medskip  \\ 
X_{(1)}^{(2)}(x^{1},x^{2},x^{3})=kx^{1}x^{3}-\delta x^{2},\medskip  \\ 
X_{(1)}^{(3)}(x^{1},x^{2},x^{3})=n\delta x^{2}-cx^{3}.%
\end{array}%
\end{equation*}%
It follows that we have the Jacobian matrix 
\begin{equation*}
J\left( X_{(1)}\right) =\left( 
\begin{array}{ccc}
p-d-\dfrac{2p}{m}T-kV & 0 & -kT\medskip  \\ 
kV & -\delta  & kT\medskip  \\ 
0 & n\delta  & -c%
\end{array}%
\right) .
\end{equation*}

In conclusion, using the Theorem \ref{MainTh}, we find the required result.
\end{proof}

\begin{remark}[Open problem]
The \textbf{Yang-Mills biologic energetical \linebreak surfaces of constant
level produced by the HIV-1 flow (\ref{HIV})} have in the system of axis $%
OTT^{\ast }V$ the implicit equations 
\begin{equation*}
\Sigma _{C}:k^{2}(V^{2}+T^{2})+(kT-n\delta )^{2}=4C,
\end{equation*}%
where $C\geq 0.$ It is obvious that the surfaces $\Sigma _{C}$ are some real
or imaginar \textbf{cylinders}. Taking into account that the family of conics%
\begin{equation*}
\Gamma _{C}:2k^{2}T^{2}+k^{2}V^{2}-2kn\delta T+n^{2}\delta ^{2}-4C=0,
\end{equation*}%
which generate the cylinders $\Sigma _{C}$, have the matrices%
\begin{equation*}
A=\left( 
\begin{array}{ccc}
2k^{2} & 0 & -kn\delta \\ 
0 & k^{2} & 0 \\ 
-kn\delta & 0 & n^{2}\delta ^{2}-4C%
\end{array}%
\right) ,
\end{equation*}%
it follows that their invariants are $\Delta _{C}=k^{4}\left( n^{2}\delta
^{2}-8C\right) ,$ $\delta =2k^{4}>0$ and $I=3k^{2}>0.$ As a consequence, we
have the following situations:

\begin{enumerate}
\item If $0\leq C<\dfrac{n^{2}\delta ^{2}}{8}$, then we have the \textbf{%
empty set} $\Sigma _{0\leq C<\frac{n^{2}\delta ^{2}}{8}}=\emptyset $;

\item If $C=\dfrac{n^{2}\delta ^{2}}{8}$, then the surface $\Sigma _{C=\frac{%
n^{2}\delta ^{2}}{8}}$ degenerates into the \textbf{straight line }%
\begin{equation*}
\Sigma _{C=\frac{n^{2}\delta ^{2}}{8}}:\left\{ 
\begin{array}{l}
T=\dfrac{n\delta }{2k}\medskip \\ 
V=0%
\end{array}%
\right. ;
\end{equation*}

\item If $C>\dfrac{n^{2}\delta ^{2}}{8}$, then the surface $\Sigma _{C>\frac{%
n^{2}\delta ^{2}}{8}}$ is a \textbf{right elliptic cylinder} of equation%
\begin{equation*}
\Sigma _{C>\frac{n^{2}\delta ^{2}}{8}}:\frac{\left( T-\dfrac{n\delta }{2k}%
\right) ^{2}}{a^{2}}+\frac{V^{2}}{b^{2}}=1,\text{ }T^{\ast }\in \mathbb{R},
\end{equation*}%
where $a<b$ are given by%
\begin{equation*}
a=\frac{\sqrt{8C-n^{2}\delta ^{2}}}{2k},\text{ }b=\frac{\sqrt{8C-n^{2}\delta
^{2}}}{k\sqrt{2}}.
\end{equation*}%
Obviously, it has as axis of symmetry the straight line $\Sigma _{C=\frac{%
n^{2}\delta ^{2}}{8}}$.
\end{enumerate}

There exist possible valuable informations for biologists contained in the
shapes of the Yang-Mills energetical constant surfaces $\Sigma _{C}$?
\end{remark}

Ileana Rodica NICOLA

\textbf{Author's adress:} University \textquotedblright
Politehnica\textquotedblright\ of Bucharest, Faculty of Applied Sciences,
Department of Mathematics I, Splaiul Independen\c{t}ei, No. 313, RO-060042
Bucharest, Romania.

\textbf{E-mail address:} nicola\_rodica@yahoo.com

\bigskip

Mircea NEAGU

\textbf{Author's adress: }Str. L\u{a}m\^{a}i\c{t}ei, Nr. 66, Bl. 93, Sc. G,
Ap. 10, Bra\c{s}ov, BV 500371, Romania.

\textbf{E-mail address:} mirceaneagu73@yahoo.com

\textbf{Place of work:} University \textquotedblright
Transilvania\textquotedblright\ of Bra\c{s}ov, Faculty of Mathematics and
Informatics


\begin{thebibliography}{99}
\bibitem{Asanov[1]} G.S. Asanov, \textit{Jet Extension of Finslerian Gauge
Approach}, Fortschritte der Physik \textbf{38}, No. \textbf{8} (1990),
571-610.

\bibitem{Garner[2]} A.L. Garner, Y.Y. Lau, D.W. Jordan, M.D. Uhler, R.M.
Gilgenbach, \textit{Implication of a Simple Mathematical Model to Cancer
Cell Population Dynamics}, Cell Prolif. \textbf{39} (2006), 15-28.

\bibitem{Luciani[3]} A.M. Luciani, A. Rosi, P. Matarrese, G. Arancia, L.
Guidoni, V. Viti, \textit{Changes in Cell Volume and Internal Sodium
Concentration in HrLa Cells During Exponential Growth and Following
Ionidamine Treatment}, Eur. J. Cell Biol. \textbf{80} (2001), 187.

\bibitem{Mir-An[4]} R. Miron, M. Anastasiei, \textit{The Geometry of
Lagrange Spaces: Theory and Applications}, Kluwer Academic Publishers, 1994.

\bibitem{Neagu Carte[5]} M. Neagu, \textit{Riemann-Lagrange Geometry on
1-Jet Spaces}, Ed. Matrix Rom, Bucharest, 2005.

\bibitem{Rodica[6]} M. Neagu, I.R. Nicola, \textit{Geometric Dynamics of
Calcium Oscillations ODEs Systems}, Balkan Journal of Geometry and Its
Applications \textbf{9}, No. \textbf{2 }(2004), 36-67.

\bibitem{TM[7]} M. Neagu, C. Udri\c{s}te, \textit{From PDEs Systems and
Metrics to Ge\-o\-me\-tric Multi-Time Field Theories}, Seminarul de Mecanic%
\u{a}, Sisteme Dinamice Diferen\c{t}iale, No. \textbf{79} (2001), Timi\c{s}%
oara, Romania.

\bibitem{Nelson[8]} P.W. Nelson, A.S. Perelson, \textit{Mathematical
Analysis of Delay Differential Equation Models of HIV-1 infection},
Mathematical Biosciences \textbf{179} (2002), 73-94.

\bibitem{Nicola Carte[9]} I.R. Nicola, \textit{Geometric Methods for the
Study of Some Complex Biological Processes}, Ed. Bren, Bucharest, 2007 (in
Romanian).

\bibitem{Obadeanu[10]} V. Ob\u{a}deanu, \textit{Sisteme Dinamice Diferen\c{t}%
iale. Dinamica Materiei Amorfe}, Editura Universit\u{a}\c{t}ii de Vest, Timi%
\c{s}oara, Romania, 2006 (in Romanian).

\bibitem{Olver[11]} P.J. Olver, \textit{Applications of Lie Groups to
Differential Equations}, Springer-Verlag, 1986.

\bibitem{Perelson[12]} A.S. Perelson, P.W. Nelson, \textit{Mathematical
Analysis of HIV-1 Dynamics in Vivo}, Siam Review \textbf{41}, No. \textbf{1}
(1999), 3-44.

\bibitem{Poincare[13]} H. Poincar\'{e}, \textit{Sur les Courbes Definies par
les Equations Diff\'{e}rentielle}, C.R. Acad. Sci. Paris \textbf{90} (1880),
673-675.

\bibitem{Reya[14]} T. Reya, S.J. Morrison, M.F. Clarke, I.L. Weissman, 
\textit{Stem Cells, Cancer, and Cancer Stem Cells}, Nature \textbf{414}
(2001), 105.

\bibitem{Saunders[15]} D. Saunders, \textit{The Geometry of Jet Bundles},
Cambridge University Press, New York, London, 1989.

\bibitem{Solyanik[16]} G.I. Solyanik, N.M. Berezetskaya, R.I. Bulkiewicz,
G.I. Kulik, \textit{Different Growth Patterns of a Cancer Cell Population as
a Function of its Starting Growth Characteristichs: Analysis by Mathematical
Modelling}, Cell Prolif. \textbf{28} (1995), 263.

\bibitem{Udr Geom Dyn[17]} C. Udri\c{s}te, \textit{Geometric Dynamics},
Kluwer Academic Publishers, 2000.

\bibitem{Udr Econ Dyn[18]} C. Udri\c{s}te, M. Ferrara, D. Opri\c{s}, \textit{%
Economic Geometric Dynamics}, Geometry Balkan Press, Bucharest, 2004.

\bibitem{Vondra[19]} A. Vondra, \textit{Symmetries of Connections on Fibered
Manifolds}, Archivum Mathematicum, Brno, Tomus \textbf{30} (1994), 97-115.
\end{thebibliography}
\end{document}